\documentclass[12pt, reqno]{amsart}
\usepackage{amsmath, amsthm, amscd, amsfonts, amssymb, graphicx, color}

\newtheorem{df}{Definition}[section]
\newtheorem{thm}[df]{Theorem}
\newtheorem{pro}[df]{Proposition}

\newtheorem{rema}[df] {Remark}

\begin{document}
\setcounter{page}{1}

\title[Dual Properties and Joint Spectra]{Dual Properties and Joint Spectra \\  for solvable Lie Algebras of 
operators}
\author{Enrico Boasso}

\begin{abstract}Given $L$  a solvable Lie Algebra of operators acting on a
Banach space $E$, we study the action of the opposite algebra of $L$,
$L'$, on $E^*$. Moreover, we extend S\l odkowski joint
spectra, $\sigma_{\delta,k}$, $\sigma_{\pi,k}$, and we study its usual spectral
properties.\end{abstract}
\maketitle
\section{Introduction}
\noindent       In [1] we defined a joint spectrum for a finite
dimensional complex solvable Lie algebras of operators $L$ acting on
a Banach space $E$ and we denoted it by $Sp(L,E)$. We also proved that $Sp(L,E)$
is a compact non void subset of $L^{2^\perp}= \{ f\in L\colon f(L^2)=0\}$.
Besides, if $I$ is an ideal of $L$, the projection property holds.
Furthermore, if $L$ is a commutative algebra, this spectrum reduces to the
Taylor joint spectrum ([5]).

                 Let $a$ be an n-tuple of commuting operators acting on $E$,
$a=(a_1,..,a_n)$. Let $a^*$ be the adjoint n-tuple of a, i.e.,
$a^*=(a_1^*,...,a_n^*)$,
where $a_i^*$ is the adjoint operator of $a_i$. Then $a^*$ is an n-tuple of
commuting operators acting on $E^*$, the dual space of $E$. If $\sigma(a)$ (respectively
$\sigma(a^*)$) denotes the Taylor joint spectrum of a (respectively $a^*$), it is
well known that $\sigma(a)=\sigma(a^*)$. If we consider a solvable non
commutative Lie algebra of operators $L$  contained in ${\mathcal L}(E)$, the space of
bounded linear maps on E, its dual, $L^*=\{x^*\colon x\in L\}$ defines a solvable
Lie subalgebra of ${\mathcal L}(E^*)$ with the opposite bracket of L. One may ask if the
joint spectra of L and $L^*$ in the sense of [1] coincide. In the solvable
non commutative case, in general, the answer is no.

                 We study this problem and prove that $Sp(L,E)$ and $Sp(L^*,E^*)$ are
related: one is obtained from the other by a traslation, i.e., $Sp(L,E)=
Sp(L^*,E^*) + c$, where $c$ is a constant. Moreover, we characterize this
constant in terms of the algebra and prove that in the nilpotent case $c=0$.

                 In the second part of our work, we study $\sigma_{\delta,k}$
and $\sigma_{\pi,k}$, the S\l odkowski spectra of [4]. We extend then to the
case of solvable Lie algebras of operators  and verify the usual spectral
properties: they are compact, non void sets and the projection property
for ideals still holds.

                 The paper is organized as follows. In section 2 we review several
definitions and results of [1]. In section 3 we study the relation between $Sp(L,E)$ and
$Sp(L^*, E^*)$, the dual property. In section 4 we extend S\l odkowski spectrum and
prove its spectral properties.
\vskip4pt
\section{Preliminaries}
       \noindent          We shall briefly recall several definitions and results
related to the spectrum of a solvable Lie algebra of operators ([1]).

                 From now on, $L$ denotes a complex finite dimensional solvable 
Lie algebra. $E$ denotes a Banach space on which $L$ acts as right continous
operators, i. e., $L$ is a Lie subalgebra of ${\mathcal L}(E)$.

                 Let f be a character of $L$ and suppose that $n=\dim\hbox{ } L$. Let us
consider the following complex, $(E\otimes\wedge L, d(f))$, where $\wedge L$
denotes the exterior algebra of L and
 
$$ 
d_p(f): E\otimes\wedge^p L\to E\otimes\wedge^{p-1} L, 
$$
\begin{align*} 
 &d_p(f)e\langle x_1\wedge\ldots\wedge x_p\rangle =\sum_{k=1}^{k=p}(-1)^{k+1}
e(x_k-f(x_k))\langle x_1\wedge\ldots \wedge\Hat {x_k}\wedge\ldots \wedge x_p\rangle\\
&+\sum_{1\le k<l \le p} (-1)^{k+l} e\langle[x_k,x_l]\wedge x_1\wedge\ldots
\wedge \Hat {x_k}\wedge\ldots\wedge \Hat {x_l}\wedge\ldots\wedge x_p\rangle, \\\end{align*}
$$\eqno(1)$$

\noindent where $\Hat{}$ means deletion. If $ p\le 0$ or $p\ge n+1$, we also define
$d_p(f)\equiv0$.

                 Let $ H_*\bigl(E\otimes\wedge L, d(f)\bigr)$ denotes
the homology of the complex  $\bigl(E\otimes\wedge L, d(f)\bigr)$.
\vskip4pt
\begin{df}Let L and E be as above. The set $\{ f\in L^*\colon f(L^2)=0 \hbox{ and }
H_*\bigl(E\otimes\wedge L,d(f)\bigr)\ne 0\}$  is the spectrum of $L$ acting
on $E$, and it is denoted by $Sp(L,E)$.\end{df}

          \begin{thm} If L is a commutative Lie algebra, Sp(L,E)
reduces to the Taylor joint spectrum.\end{thm}

\begin{thm}Sp(L,E) is a compact non void subset of $L^*$.\end{thm}
\vskip4pt
         \begin{thm} (Projection property) Let I be an ideal ol L and
$\pi$ the projection map from $L^*$ onto $I^*$, then
$$ Sp(I,E)= \pi(Sp(L,E)).$$\end{thm}
\vskip4pt
                 As in [1], we consider an n-1 dimensional ideal of $L$,
$L_{n-1}$, and we decompose $ E\otimes\wedge^p L$ in the following
way

$$E\otimes\wedge^p L=\bigl(E\otimes\wedge^p L_{n-1}\bigr)\oplus
\bigl( E\otimes\wedge^{p-1} L_{n-1}\bigr)\wedge\langle x_n\rangle, $$

\noindent where $ x_n \in L $ and it is such that $L_{n-1}\oplus\langle x_n\rangle=L$ 

                 If $\tilde f$ denotes the restriction of $f$ to $L_{n-1}$,
we may consider the complex $\bigl(E\otimes\wedge^p L_{n-1}, d(\tilde f)\bigr)$

 As $L_{n-1}$ is an ideal of codimension 1 of $L$, we may decompose the
operator $d_p(f)$ as follows

$$ d_p(f)\colon E\otimes\wedge^p L_{n-1}\to E\otimes\wedge^{p-1} L_{n-1},$$

$$ d_p(f)=\tilde d_p(\tilde f), \eqno(2)$$

$$ d_p(f)\colon E\otimes\wedge^{p-1} L_{n-1}\wedge\langle x_n\rangle\to
E\otimes\wedge^{p-1} L_{n-1}\oplus E\otimes\wedge^{p-2} L_{n-1}\wedge \langle
x_n\rangle, $$

$$  d_p(f)(a\wedge\langle x_n\rangle)= (-1)^p L_p(a)+
(\tilde d_{p-1}(\tilde f)(a)) \wedge\langle x_n\rangle, \eqno(3) $$

\noindent where $ a  \in E\otimes\wedge^{p-1} L_{n-1}$ and $L_p$ is the the
bounded linear endomorphism defined on $E\otimes\wedge^{p-1} L_{n-1}$  by

\begin{align*} 
 L_p e\langle x_1\wedge\ldots\wedge x_{p-1}\rangle
&= -e(x_n-f(x_n))\langle x_1\wedge\ldots \wedge x_{p-1}\rangle\\
&+\sum_{1\le k \le p-1} (-1)^{k+1} e\langle[x_k,x_n]\wedge x_1\wedge\ldots
\wedge \Hat {x_k}\wedge\ldots\wedge x_{p-1}\rangle, \\ \end{align*}
$$\eqno(4)$$
\noindent where $\Hat{}$ means deletion and $x_i$ $(1\le i\le p-1)$ belongs to $L_{n-1}$.
\vskip4pt
                 Now we consider the following morphism defined in [3,2].

                 Let $\theta(x_n)$ be the derivation of $\wedge L$ that extends the map $ ad(x_n)$, \par
 $ad(x_n)(y)= [x_n,y]$ ($y\in L$),

$$ \theta(x_n)\langle x_1\wedge\ldots\wedge x_p\rangle=
\sum_{i=1}^p\langle x_1\wedge\ldots\wedge ad(x_n)(x_i)\wedge\ldots
\wedge x_p\rangle. \eqno(5)                      $$

$ \theta(x_n)$ satisfies :
$$
\theta(x_n)(ab)= (\theta(x_n)a)b + a(\theta(x_n)b), \eqno(6) 
$$

$$
\theta(x_n)w= w\theta(x_n), \eqno(7)
$$

\noindent where $w\colon \wedge L\to \wedge L$ is the map
$$
w(\langle x_1\wedge\ldots\wedge x_p\rangle) = (-1)^p\langle x_1\wedge\ldots\wedge x_p\rangle.
\eqno(8)
$$

\indent Let $u$ belong to $\wedge L$ and $\epsilon(u)$ be the following endomorphism defined of $\wedge L$:
$\epsilon(u)v= u\wedge v$ ($v\in\wedge L$). As $(\wedge L)^*$ may be
identified with $\wedge L^*$, let $\iota(u)$ be the dual map of $\epsilon(u)$,
$(\iota(u)\colon \wedge L^*\to\wedge L^*)$.

\indent Besides, we consider $\theta^*(x_n)$, the dual map of $-\theta(x_n)$.\par
\indent As $\epsilon(u\wedge v) = \epsilon(u)\epsilon(v)$, $\iota(w\wedge z)=
\iota(w)\iota(z)$ ($u$, $v\in\wedge L$, $w$, $z\in\wedge L^*$).\par
\indent As in [3,7] we define an isomorphism $\rho$
$$\rho\colon \wedge L^*\to\wedge L,$$
$$\rho(a) = \iota(a).w, \eqno(9)$$
\noindent where $a$ $\in \wedge L^*$, $w =\langle x_1\wedge\ldots\wedge x_{n-1}\wedge x_n\rangle$
and $\{ x_1\ldots x_n\}$ is a basis of $L_{n-1}$. Note that  $\rho$ applies $\wedge^p
L^*$ isomorphically onto $\wedge^{n-p} L$.\par
\vskip4pt
\section{The Dual Property}
\noindent Let $L$ and $E$ be as in section 2. Let $E^*$ be the space of
continous functionals on $E$. Let $L^{'}$ be the solvable Lie algebra 
defined as follows: as vector space, $L=L^{'}$, and the bracket of $L^{'}$ is
the opposite of the one of $L$; that is: $ [x,y]^{'} = -[x,y] =[y,x]$.
$L^{'}$is a complex finite dimensional solvable Lie algebra and $L^{2^\perp} =
L^{'{2^\perp}}$.\par
\indent As $L$ acts as right continous operators on $E$, the space $L^*=
\{x^*\colon x\in L\}$ has the Lie structure of the algebra $L^{'}$ and acts as 
right continous operators on $E^*$.\par
\indent Observe that in the definition of $\epsilon$, $\iota$  and $\rho$,
we only consider the structure of $L$ as vector space. As $L$ and $L^{'}$
coincide as vector spaces, then $\wedge L = \wedge L^{'}$ and we may consider
$$
\rho\colon\wedge L^*\to\wedge L^{'}.
$$

\indent If $L_{n-1}$ is an ideal of codimension 1 of $L$, $L^{'}_{n-1} = L_{n-1}$
is an ideal of codimension 1 of $L^{'}$. Moreover, if $x_n\in L$ is such that
$L_{n-1}\oplus\langle x_n\rangle =L$, then $L^{'}_{n-1}\oplus\langle x_n\rangle = L^{'}$.\par
\indent Let $\theta^{'}(x_n)$ be the derivation of $\wedge L^{'}$ that extends
the map $ad(x_n)$, $ad(x_n)(y) =[x_n,y]^{'}$. By [2, Chapter V, Section 3], there exist a basis
of $L$, $\{x_i\}_{ 1\le i\le n}$, such that 
$$
[x_j, x_i]= \sum_{h=1}^i c_{ij}^h x_h \hbox{    (i \textless j)} \eqno(10)
$$
and $L_{n-1}$ has the basis $\{x_i\}_{1\le i\le n-1}$. If $w= \langle x_1\wedge\dots\wedge x_n\rangle$,
then
\begin{align*}
\theta(x_n)w&= \theta(x_n)\langle x_1\wedge\dots\wedge x_n\rangle\\
            &=\sum_{i=1}^n \langle x_1\wedge\dots\wedge
x_{i-1}\wedge [x_n,x_i]\wedge\dots\wedge x_n\rangle\\
 &=(\sum_{i=1}^{n-1} c_{in}^i)  \langle x_1\wedge\dots\wedge x_n\rangle \\
&=(\hbox{\rm trace } \theta(x_n)(x_n)) w.\\
\end{align*}

\indent As in [3,7], if $a\in \wedge L^*$, we have
$$
\rho\theta^*(x_n)= \theta(x_n)\rho(a) - \iota(a)\theta(x_n)w. \eqno(11)
$$
Then
$$ \rho\theta^*(x_n)= \theta\rho - (\hbox{\rm trace }  ad(x_n))\rho. \eqno(12)$$

As $L^{'}$ has the opposite bracket of $L$,
$$
\rho\theta^*(x_n)= -(\theta^{'}(x_n) + \hbox{\rm trace } ad(x_n))\rho. \eqno(13)
$$
\indent Let us consider the maps $1_{E^*}\otimes\theta(x_n)$, 
$1_{E^*}\otimes\rho$, $1_{E^*}\otimes\theta(x_n)$, $1_{E^*}\otimes\theta^{'}(x_n)$
and let us still denote then by $\theta^*(x_n)$, $\rho$, $\theta(x_n)$,
$\theta^{'}(x_n)$, respectively. We observe that formula (11),
(12), (13) remain true.\par
\indent Let us decompose, as in section 2, $E^*\otimes\wedge^p L^*$ 
( $E^*\otimes\wedge^{n-p} L^{'}$, respectively) as the sum
$$
 E^*\otimes\wedge^p L^*_{n-1}\oplus E^*\otimes\wedge^{p-1} L_{n-1}^*\wedge\langle x_n\rangle,
$$
$$
E^*\otimes\wedge^{n-p} L^{'}_{n-1}\oplus E^*\otimes\wedge^{n-p-1}L_{n-1}^{'}\wedge\langle x_n\rangle \hbox{(respectively)},
$$
where $L^*_{n-1}$ is the subspace of $L^*$ generated by $\{y_j\}_{1\le j\le n-1}$
($\{y_j\}_{ 1\le j \le n}$ is the dual basis of $\{x_j\}_{ 1\le j\le n}$).\par
\indent A standard calculation shows the following facts:
$$
\rho(E^*\otimes\wedge^p L^*_{n-1})= (E^*\otimes\wedge^{n-p-1}L^{'}_{n-1})\wedge \langle x_n\rangle,\eqno(14)
$$
$$
\rho(E^*\otimes\wedge^{p-1}L_{n-1}\wedge\langle y_n\rangle)=E^*\otimes\wedge^{n-p} L^{'}_{n-1}, \eqno(15)
$$
$$
\rho\vert E^*\otimes\wedge^p L^*_{n-1} = \rho_{n-1}\wedge\langle
x_n\rangle, \eqno(16)
$$
$$
\rho\vert E^*\otimes\wedge^{p-1} L_{n-1}\wedge\langle y_n\rangle
=\rho_{n-1}(-1)^{n-p}, \eqno(17)
$$
where $\rho_{n-1}$ is the isomorphism associated to the algebra $L_{n-1}$.
\vskip4pt 

\begin{pro} Let $L$, $L{'}$, $E$, $L_{n-1}$ and $\rho$ be as above
and let $f$ belong to $L^{2^\perp}$ and $g$ to ${L^{'}}^{2^\perp}$ such that
$f(x_n) + \hbox{\rm trace }\theta(x_n) =g(x_n)$.
Then, the following diagram commutes:

$$
\CD
E^*\otimes\wedge^p L^*_{n-1} @>L^*_{p+1}>>E^*\otimes\wedge^p L^*_{n-1}\wedge\langle y_n\rangle \\
@V\rho VV    @VV\rho V \\
E^*\otimes\wedge^{n-p-1} L_{n-1}^{'}\wedge\langle x_n\rangle@>L^{'}_{n-p}>>E^*\otimes\wedge^{n-p-1} L^{'}_{n-1}, \\ 
\endCD
$$
         
\noindent where  $L^{'}_{n-p}$ is the operator involved in the definition of $d^{'}_{n-p}(g)$
and $L^*_{p+1}$ is the adjoint operator of $L_{p+1}$, which is involved in the definition
of $d_{p+1}(f)$ (see (3), (4)).\end{pro}

\begin{proof}
By (4), (5)
$$
L_{p+1} = -(x_n-f(x_n)) + \theta(x_n).
$$
Then,
$$
 L^*_{p+1}= -(x_n^*-f(x_n)) - \theta^*(x_n).
$$
As the bracket of $L^{'}$ is the opposite of the one of $L$, by (4) and (5),
$$
L^{'}_{n-p} = -(x^*_n - g(x_n)) + \theta(x_n).
$$

Then
$$
 L_{n-p}^{'}\rho = -(x_n^*-g(x_n))\rho + \theta^{'}(x_n)\rho.
$$
On the other hand, by (13)

\begin{align*} \rho L^*_{p+1} &= -(x^*_n-f(x_n))\rho - \rho\theta^*(x_n)\\
&= -(x^*_n - f(x_n))\rho  + \theta^{'}(x_n)\rho + \hbox{ \rm trace }f(x_n)\rho\\
&= -(x^*_n -(f(x_n) + trace(x_n)))\rho + \theta^{'}\rho\\
&= -(x^*_n - g(x_n))\rho + \theta^{'}(x_n)\rho\\
&= L^{'}_{n-p}\rho.\\ \end{align*}
\end{proof}

\begin{thm} Let $L$, $L^{'}$, $E$,  and $\rho$ be as above. Let 
$\{x_i\}_{1\le i\le n}$ be the basis of $L$ defined in (10). Let $f$ belong to
$L^{2^\perp}$ and $g$  to ${L^{'}}^{2\perp}$ such that
$$
g = f + (\hbox{\rm trace }\tilde\theta(x_1),\dots , \hbox{\rm trace }\tilde\theta(x_n)),
$$ 
\noindent where $\tilde\theta(x_i)$ is the restriction of $\theta(x_i)$ to $L_i$,
the ideal generated by $\{x_l\}_{1\le l\le i}$ ($1\le i\le n$).
Then, if we consider the adjoint complex of $(E\otimes\wedge L, d(f))$
and the complex $(E\otimes\wedge L^{'}, d(g))$, for each p,
$$
d^{'}_{n-p}(g)\rho w = \rho d^*_{p+1}(f).
$$
That is, the following diagram commutes:
$$
\CD
E^*\otimes\wedge^p L @>d^*_{p+1}>> E^*\otimes\wedge^{p+1} L\\
@V\rho w VV @VV\rho V\\
E^*\otimes\wedge^{n-p} L^{'} @>d^{'}_{n-p}>> E^*\otimes\wedge^{n-p-1} L^{'},\\
\endCD
$$
\noindent where $w$ is the map of (8).\end{thm}

\begin{proof}
By means of an induction argument the proof may be derived from Proposition 3.1.\end{proof}

\begin{thm} Let $L$, $L^{'}$ and $E$ be as above. If we consider
the  basis of $L$  $\{x_i\}_{1\le i\le n}$ defined in (10)), then in terms of the dual basis
in $L^* (= L^{'})$ we have
$$
Sp(L,E) + (\hbox{\rm trace }\tilde\theta(x_1),\dots ,\hbox{ \rm trace }\tilde\theta(x_n)) = Sp(L^{'},E^*),
$$
\noindent where $\tilde\theta(x_i)$ is as in Proposition 3.1.\end{thm}
\begin{proof}
Is a consequence of Theorem 3.2 and [4, Lemma 2.1]. \end{proof}

\begin{thm} If $L$ is a nilpotent Lie Algebra, then
$$
Sp(L,E) = Sp(L^{'},E^*).
$$
\end{thm}
\begin{proof}
By [2, Chapter V, Section 1], the basis $\{x_i\}_{1\le i\le n}$ of (10) may be chosen such that
$$
[x_j,x_i] = \sum_{h=1}^{i-1} c^h_{ji} x_h\hbox{     (j\textgreater i)}.
$$
Then $\tilde\theta(x_i) = 0$.
\end{proof}

\begin{rema}\rm    
Let $L$ be a solvable Lie Algebra. Let $n=dim\, (L)$ and $k=\ dim\,  (L^2)$. By [2, Chapter V, Section 3],
the basis of (10) may be chosen such that $\{x_j\}_{1\le j\le k}$ generates
$ L^2$ . As $L^2$ is a nilpotent ideal of $L$, $\hbox{ \rm trace }\theta(x_i) = 0$, 
$1\le i\le k $. Then, we have the following proposition.
\end{rema}

\begin{pro} Let $L$, $L^{'}$, and $E$ be as usual. Let us consider a 
basis of $L$ as in the previous remark. Then
$$
Sp(L,E) + (0,\ldots , 0,\hbox{ \rm trace }\tilde\theta(x_{k+1}),\ldots ,\hbox{\rm trace }\tilde\theta(x_n)) = Sp(L^{'},E^*)
$$
\end{pro}

\indent Then, as we have seen, if $L$ is a nilpotent Lie algebra, the dual
property is essentially the one of the commutative case. However, if $L$ 
is a solvable non nilpotent  Lie algebra, this property fails. For example,
if $L$ is the algebra
$$
L = \langle x_1\rangle\oplus\langle x_2\rangle,\,\,\,\, \hbox{       } [x_2, x_1] = x_1,
$$
then,
$$
\hbox{ trace }\theta(x_2) = 1,\,\,\,  \hbox{           }   \hbox{ trace }\theta(x_1) =0
$$
and
$$
Sp(L,E) + (0,1) = Sp(L^{'},E^*)
$$

\section{The extension of Slodkowski joint spectra}

\noindent Let $L$ and $E$ be as in section 2. We give an homological version of  S\l odkowski
spectra insteed of the cohomological one of [4].\par
\indent Let $\varSigma_p(L,E)$ be the set $\{ f \in L^{2^\perp} \colon
H_p(E\otimes\wedge L, d(f))\ne 0\}$ = $\{ f \in L^{2^\perp} \colon R(d_{p+1}(f))\ne
Ker(d_p(f))\}$.\par
\begin{df} Let $L$ and $E$ be as in section 2.
$$
\sigma_{\delta , k}(L,E) = \bigcup\limits_{0\le p\le n} \varSigma_p(L,E),
$$
$$
\sigma_{\pi ,k} = \bigcup\limits_{k\le p\le n} \varSigma_p (L,E)\bigcup\{ f\in L^{2^\perp} \colon R(d_k(f)) \hbox{ is closed }\},
$$
where $ 0\le k\le n$.\end{df}

\indent We shall see that $\sigma_{\delta ,k} (L,E)$ and $\sigma_{\pi ,k}(L,E)$ are
compact non void subsets of $L^*$ and that they verify the projection property for
ideals.\par
\indent Observe that $\sigma_{\delta , n} (L,E)$  = $\sigma_{\pi , 0} (L,E)$ = $Sp(L,E)$.\par
\indent Let us show that $\sigma _{\delta ,k} (L,E)$ has the usual properties of a
spectrum.\par

\begin{thm} Let $L$ and $E$ be as usual. Then $\sigma_{\delta,k} (L,E)$
is a compact subset  of $L^*$.\end{thm}
\begin{proof}
As $\sigma_{\delta ,k} (L,E)$ is contained in $Sp(L,E)$, it is enough to prove 
that $\sigma_{\delta ,k} (L,E)$ is closed in $L^*$.
Let us consider the complex $(E\otimes\wedge^p L, d_k (f))$ $( 0\le k\le p+1)$.   
$$
E\otimes\wedge^{p+1} L \xrightarrow{d_{p+1}(f)} E\otimes\wedge^p L\to \cdots\to E\otimes L \xrightarrow{d_1(f)} E\to 0.
$$
This complex is a parameterized chain complex of Banach spaces on $L^{2^\perp}$
in the sense of [5, Definition 2.1]. By [5, Theorem 2.1], $\{ f\in L^{2^\perp}\colon(E\otimes\wedge^kL,d_k(f))$ $(0\le k\le p+1)$ is not exact$\}$
is a closed set of $L^{2^\perp}$, and then in $L^*$. However, this set is
exactly $\sigma_{\delta ,k}(L, E)$.\end{proof}

\indent Let $L_{n-1}$ be an ideal of codimension 1 of $L$ and let $x_n$ in $L$ be
such that $L_{n-1}\oplus\langle x_n\rangle = L$. As in [5] and [1], we consider
a short exact sequence of complex.
Let $f$ belong to $L^*$ and denote $\tilde f$ its restriction to $L_{n-1}$. Then
$$
0\to (E\otimes\wedge L_{n-1},\tilde d (\tilde f))\xrightarrow{i}(E\otimes\wedge L,d(f))\xrightarrow{p}(E\otimes\wedge L_{n-1},\tilde d(\tilde f))\to 0, \eqno(18)
$$
where $p$ is the following map.
As in section 2 we decompose $E\otimes\wedge^k L$,
$$
E\otimes\wedge^k L = E\otimes\wedge^k L_{n-1}\oplus E\otimes\wedge^{k-1}L_{n-1}\wedge\langle x_n\rangle,
$$
$$
p(E\otimes\wedge^k L_{n-1}) =0,
$$
$$
p(e\langle x_1\wedge\dots\wedge x_{k-1}\wedge x_n\rangle) = {(-1)}^{k-1}e\langle x_1\wedge\dots\wedge x_{k-1}\rangle,
$$
\noindent where $x_i \in L_{n-1}$, $1\le i\le k-1$.\par

\begin{rema}
\rm (i) $d(f)\vert E\otimes\wedge L_{n-1} = \tilde d(\tilde f)$.
\par
\noindent (ii) As in [1], $H_p(E\otimes\wedge L,d(f)) = Tor_p^{(UL)} (E, C(f))$.
\par
\noindent (iii) As in [5] and [1], we have a long exact sequence of $ U(L)$ modules, where
$U(L)$ is the universal algebra of L
$$
\to  H_p(E\otimes \wedge L_{n-1}, \tilde d(\tilde f))\xrightarrow{ i_*} H_p(E\otimes\wedge L, d(f))\xrightarrow{ p_*}
$$
$$
\to H_{p-1} (E\otimes \wedge L_{n-1}, \tilde d(\tilde f))\xrightarrow{\delta_{* p-1}} H_{p-1} (E\otimes \wedge L_{n-1},\tilde d(\tilde f))\to,
$$
where $\delta{_* p} $ is the connecting operator.\par

\noindent (iv) As in [1], we observe that if we regard the $U(L) $ module $H_P(E\otimes\wedge
L_{n-1},\tilde d(\tilde f))$  as $U(L_{n-1})$ module, then we obtain $Tor_p^{U(L_{n-1})}(E,C(\tilde f))$.
Then, as $U(L_{n-1})$ is a subalgebra with unit of $U(L)$,
$$
\tilde f \in \sum_p (L_{n-1},E) \hbox{ if and only if }H_p(E\otimes\wedge L_{n-1},\tilde d(\tilde f))\ne 0 \hbox{ as U(L) module}.
$$
\end{rema}

\begin{pro} Let $L$, $L_{n-1}$, $E$, $f$, and $\tilde f$ be
as above. Then, if $f$ belongs to $\varSigma_p(L,E)$, $\tilde f$ belongs to
$\varSigma_{p-1}(L_{n-1}, E) \bigcup \varSigma_p(L_{n-1}, E)$.\end{pro}
\begin{proof}
By Remark 4.3 (iv), $\tilde f$ $\not\in $ $\varSigma_p(L_{n-1},E)\bigcup\varSigma_{p-1}(L_{n-1},E)$
if and only if $H_i (E\otimes\wedge L_{n-1}, \tilde d(\tilde f)) = 0 $ 
as $U(L)$ module ($i =p$, $p-1$ ).
By Remark 4.3 (iii), $H_p(E\otimes\wedge L, d(f)) = 0$, i. e., $f\not\in \varSigma_p(L,E)$.
\end{proof}
\begin{pro} Let $L$, $L_{n-1}$ and $E$ be as usual. Let $\varPi\colon L^* \to L^*_{n-1}$
be the projection map. Then
$$
\varPi(\sigma_{\delta ,k}(L,E))\subseteq \sigma_{\delta ,k}(L_{n-1},E).
$$
\end{pro}
\begin{proof}
Is a consequence of Proposition 4.4.
\end{proof}
\begin{pro} Let $L$, $L_{n-1}$ and $E$ be as usual and
consider $\varPi$ as in Proposition 4.5. Then,
$\sigma_{\delta ,k}(L_{n-1},E)$  = $\varPi(\sigma_{\delta ,k},(L,E))$.
\end{pro}
\begin{proof}
By Proposition 4.5 it is enough to show that
$$
\sigma_{\delta ,k}(L_{n-1}, E) \subseteq \varPi(\sigma_{\delta ,k}(L,E)).
$$  
By refining an argument of [1] and [5], we shall see that if $\tilde f$ 
belongs to $\varSigma_p(L_{n-1}, E)$, then there is an extension of $\tilde f$
to $L^*$,$f$, such that $f \in \varSigma_p(L,E)$.\par
First of all, as $\tilde f$ $\in \varSigma_p(L_{n-1}, E)\subseteq Sp(L,E)$,
by [1, Theorem 3], if $g$ is an extension of $\tilde f$ to $L^*$, $g(L^2) =0$, i.e., $g\in L^{2^\perp}$.\par
Let us suppose that our claim is false, equivalently, $H_p(E\otimes\wedge L,d(f)) = 0 $,
$\forall f \in L^{2^\perp} , \varPi(f) = \tilde f$.\par
Let us consider the connecting map associated to the long exact sequence
of Remark 4.3 (iii), 
$$
\delta_{* p}\colon H_p(E\otimes\wedge L_{n-1},\tilde d(\tilde f))\to H_p(E\otimes\wedge L_{n-1},\tilde d (\tilde f)).
$$
As $H_p(E\otimes\wedge L, d(f)) =0, \forall f \in L^{2^\perp}$, $\varPi(f) =\tilde f$, $\delta_{* p}(f)$ is a surjective map.\par
Let $k$ $\in E\otimes\wedge^p L_{n-1}$ be such that $\tilde d(\tilde f)(k) =0$.
Then, it is well knowm that if $m\in E\otimes\wedge^{p+1} L$ is such that
$ p(m) = k$, $ \delta_{* p}([k]) = [d_{p+1}(f)(m)]$, where $p$ is the map
defined in (18).
Let us consider $m =k\wedge x_n$ $\in E\otimes\wedge^p L_{n-1}\wedge\langle x_n\rangle$.
Then, $p((-1)^p m) = k$ and
$$
\delta_{* p}(f)([k])=(-1)^p [d_{p+1}(f)(m)].
$$
Since,
\begin{align*} d_{p+1}(f)(m)&= d_{p+1}(f)(k\wedge\langle x_n\rangle)\\ 
&= (\tilde d(\tilde f)(k))\wedge\langle x_n\rangle + (-1)^{p+1} L_{p+1}(k)\\
&= (-1)^{p+1}L_{p+1}(k) \in (E\otimes\wedge^p L_{n-1}),\\ \end{align*}
$$
\delta_{* p}(f)[k] = -[L_{p+1}(k)].
$$
Moreover, by Equations (4)and (5), $ L_{p+1}(k) = -k(x_n - f(x_n) + \theta(x_n)(k)$.
Then

\begin{align*} 0&= \tilde d_p(\tilde f)(d_{p+1}(f)(m))  \\
&= \tilde d_p(\tilde f)(-k(x_n - f(x_n)) + \tilde d_p(\tilde f)\theta (x_n)(k)\\
&= (-\tilde d_p(\tilde f)(k))(x_n - f(x_n)) + \tilde d_p(\tilde f)\theta(x_n)(k)\\
&= \tilde d_p(\tilde f)\theta(x_n)(k),\\ \end{align*}

\noindent which implies that,
$$
\delta_{* p}(f) = [k](x_n - f(x_n)) - [\theta(x_n)k].
$$
Let us consider the complex of Banach spaces and maps
$$
E\otimes\wedge^{p+1} L\xrightarrow{d_{p+1}(f)} E\otimes\wedge^p L \xrightarrow{d_p(f)} E\otimes\wedge^{p-1} L.\eqno(19)
$$
Then, this is an analytically parameterized complex of Banach spaces on $\mathbb{C}$, which is
exact $\forall f\in L^{2^\perp}, \varPi(f) =\tilde f$, and exact at $\infty$ ([5, Section 2]).\par
\indent As $\delta_{* p}$ differs by a constant term of the connecting map of
[5, Lemma 1.3], the argument of [5, Lemma 3.1] still applies to the complex (19). Then
$H_p(E\otimes\wedge L_{n-1}, \tilde d(\tilde f)) = 0$ as $U(L)$ module.
By Remark 4.3 (iv) we finish the proof.
\end{proof}
\begin{thm} Let $ L$ and $E$ be as usual. Let $I$ be an ideal of $L$. Then
$$
\sigma_{\delta k}(I,E) = \varPi (\sigma_{\delta k}(L,E)),
$$
where $\varPi$ denotes the projection map.\end{thm}
\begin{proof}
By [2, Chapter V, Section 3], Proposition 5 and an inductive argument, we conclude the proof of the theorem.
\end{proof}
\begin{thm} Let $L$ and $E$ be as usual. Then $\sigma_{\delta k}(L,E)$
is a non void set of $L^*$.\end{thm}
\begin{proof}
It is a consequence of [2, Chapter V, Section 3],Theorem 5 and the one dimensional case.
\end{proof}
\begin{thm}Let $L$ and $E$ be as usual. Let  $L^{'}$, $\{x_i\}_{1\le i\le n}$
 and $\tilde\theta(x_i)$ ($1\le i \le n$) be as in Theorem 3.2.
Then, in terms of the dual basis of $\{ x_i\}_{1\le i\le n}$, \par
\noindent {\rm (i)}  $ \sigma_{\delta k}(L,E) + (\hbox{\rm trace }\tilde\theta(x_1),\dots , \hbox{\rm trace }\tilde\theta(x _n)) = \sigma_{\pi k}(L^{'},E^*)$.\par
\noindent {\rm (ii)} $\sigma_{\pi k}(L,E) = \sigma_{\delta n-k}(L^{'},E) -(\hbox{\rm trace }\tilde\theta(x_1), \dots , \hbox{\rm trace }\tilde\theta(x_n))$.
\end{thm}
\begin{proof}
It is a consequence of [4, Lemma 2.1] and Theorem 3.2.\end{proof}
\begin{thm}Let $L$ and $E$ be as usual. Then $\sigma_{ \pi k}(L,E)$
is a compact subset of $L^*$ and if $I$ is an ideal of $L$ and $\varPi$ is  the
projection map from $L^*$ onto $I$, then
$$
\sigma_{\pi k}(L,E) =\varPi (\sigma_{\pi k}(L,E)).
$$
\end{thm}
\begin{proof} It is a consequence of Theorems 4.2, 4.7, 4.8 and 4.9.\end{proof}

\begin{rema}\rm In [1] it was proved that the projection property for
subspaces that are not ideals fails for $Sp(L,E)$. As $\sigma_{\pi 0}(L,E)$=
$\sigma_{\delta n}(L,E)$ =$ Sp(L,E)$, the same result remains true, in
general, for $\sigma_{\delta k}(L,E)$  and $\sigma_{\pi k}(L,E)$.
\end{rema}
\begin{thm} Let $L$ be a nilpotent Lie algebra and let $E$, $L^{'}$, $\{x_i\}_{1\le i \le n}$
and $\tilde\theta(x_i)$ ($1\le i\le n$) be as in Theorem 3.2. Then,\par
\noindent {\rm (i)} $\sigma_{\delta k}(L,E) =\sigma_{\pi k}(L^{'},E^*)$.\par
\noindent {\rm (ii)} $\sigma_{\pi k}(L,E) = \sigma_{\delta n-k}(L^{'},E^*)$.\par
\end{thm}
\begin{proof} 
It is a consequence of Theorem 4.9 and the proof of Theorem 3.4.\end{proof}

\bibliographystyle{amsplain}

\vskip.5cm

Enrico Boasso\par
E-mail address: enrico\_odisseo@yahoo.it

\end{document}